\documentclass[11pt]{amsart}
\usepackage{amsfonts, amsmath, enumerate, amssymb, xypic}
\newtheorem{theorem}{Theorem}[section]
\newtheorem{lemma}[theorem]{Lemma}

\newtheorem{proposition}[theorem]{Proposition}
\newtheorem{remark}[theorem]{Remark}
\newtheorem{definition}[theorem]{Definition}

\newcommand{\ncom}{\newcommand}
\ncom{\lrar}{\longrightarrow}
\ncom{\rar}{\rightarrow}
\ncom{\ov}{\overline}
\ncom{\m}{\mbox}
\ncom{\sta}{\stackrel}
\ncom{\comx}{{\mathbb C}}
\ncom{\A}{{\mathbb A}}
\ncom{\Cl}{{\mathbb C}}
\ncom{\Z}{{\mathbb Z}}
\ncom{\Q}{{\mathbb Q}}
\ncom{\R}{{\mathbb R}}
\ncom{\G}{{\mathbb G}}
\ncom{\al}{\alpha}
\ncom{\p}{{\mathbb P}}
\ncom{\E}{{\mathbb E}}
\ncom{\N}{{\mathbb N}}
\ncom{\K}{{\mathbb K}}
\ncom{\X}{{\mathbb X}}
\ncom{\f}{\frac}
\ncom{\cA}{{\mathcal A}}
\ncom{\cB}{{\mathcal B}}
\ncom{\cX}{{\mathcal X}}
\ncom{\cO}{{\mathcal O}}
\ncom{\cW}{{\mathcal W}}
\ncom{\cP}{{\mathcal P}}
\ncom{\cS}{{\mathcal S}}
\ncom{\cM}{{\mathcal M}}
\ncom{\cC}{{\mathcal C}}
\ncom{\cT}{{\mathcal T}}
\ncom{\cF}{{\mathcal F}}
\ncom{\cN}{{\mathcal N}}
\ncom{\cJ}{{\mathcal J}}
\ncom{\cK}{{\mathcal K}}
\ncom{\cV}{{\mathcal V}}
\ncom{\cZ}{{\mathcal Z}}
\ncom{\cU}{{\mathcal U}}
\ncom{\cSU}{{\mathcal S \mathcal U}}
\ncom{\cG}{{\mathcal G}}
\ncom{\cQ}{{\mathcal Q}}
\ncom{\cY}{{\mathcal Y}}
\ncom{\cE}{{\mathcal E}}
\ncom{\what}{\widehat}
\ncom{\delbar}{\overline{\partial}}

\ncom{\eop}{{\hfill $\Box$}}

\def \cL{\mathcal L}

\def \H{\rm H}
\def \C{\rm C}
\def \CH{\rm {CH}}

\def \Sym{\rm {Sym}}
\def \Pic{\rm {Pic}}
\def \M{\rm M}
\def \bPic{\mathbf {Pic}}
\def \P{\rm P}
\def \NS{\rm {NS}}
\def \Q{\rm Q}
\def \J{\rm J}
\def \ra{\rightarrow}
\def \Br{\rm {Br}}
\def \O{\mathcal O}
\def \limn{\underset {\infty \leftarrow n}  {\hbox {lim}}}
\def \res{respectively}
\def \Sm{\rm Sm}

\def \rmd{\rm d}
\def \rmD{\rm D}
\def \rmX{\rm X}
\def \rmY{\rm Y}
\def \rmp{\rm p}
\def \rmU{\rm U}
\def \rmN{\rm N}
\def \rmE{\rm E}
\def \rmW{\rm W}

\def \Hilb{\rm {Hilb}}
\def \Hom{{\mathcal H}om}
\def \d{\it d}
\def \i{\it i}
\def \t{\it t}
\def \p{\rm p}
\def \rmA{\rm A}
  
\begin{document}
\baselineskip=16pt
\let\thefootnote\relax

\title[Brauer groups of schemes associated to symmetric powers of curves]{Brauer groups of  schemes associated to symmetric powers of smooth projective curves  in arbitrary characteristics}


\author[J. N. Iyer]{Jaya NN  Iyer}

\address{The Institute of Mathematical Sciences, CIT
Campus, Taramani, Chennai 600113, India}
\email{jniyer@imsc.res.in}

\author[R. Joshua]{Roy Joshua}
\address{Department of Mathematics, The Ohio State University, Columbus, Ohio, 43210, USA}
\email{joshua.1@math.osu.edu}

\footnotetext{Mathematics Classification Number: 14C25, 14F20, 14F22, 14D20, 14D23 }

\begin{abstract}
In this paper we show that the $\ell^n$-torsion part of the cohomological Brauer groups of certain schemes associated to symmetric powers of a projective
smooth curve over a separably closed field $k$ are isomorphic, when $\ell$ is invertible in $k$. The schemes considered are the Symmetric powers themselves, 
then the corresponding Picard schemes and also certain Quot-schemes. We also obtain similar results for Prym varieties associated
 to certain finite covers of such curves.
\end{abstract}
\maketitle

\setcounter{tocdepth}{1}
\tableofcontents


\section{Introduction}
The theory of Brauer groups has a long and rich history. In recent years there seems to 
be a renewed interest in this area, and there have been several important developments
in the last 10-15 years using sophisticated techniques. It is also one of the few topics
that can be studied both from a complex algebraic geometry point of view as well as using
algebraic tools, notably \'etale cohomology. 
\vskip .2cm
The present paper originated with the authors being intrigued by the recent paper of Biswas, Dhillon and
Hurtubise (see \cite{BDH}). The above paper is written purely in the context of complex algebraic
geometry, though their main results make sense in any characteristics. One of the results we
obtain here is an extension of their results valid over separably closed
   fields of arbitrary characteristics using motivic and \'etale cohomology techniques.
\vskip .2cm
Let $k$ denote any separably closed field of arbitrary characteristic and let $\C$ denote a projective smooth curve over $k$.
 {\it In this case we will make the standing assumption that the curve $\C$
has a $k$-rational point.}
Then $\Sym^d(\C)$ will denote the $\rmd$-fold symmetric power of the curve, which is a projective smooth variety of dimension $\rmd$. 
$\Pic^d(\C)$ will denote the Picard-scheme of isomorphism classes of line bundles of degree $\rmd$ on $\C$ and $\Q(r, d)$ will
denote the quot-scheme parametrizing degree $\rmd$ quotients of the $\O_{\C}$-module $\O_{\C}^{\oplus r}$.  
Similarly $\Hilb^d(\C)$ will denote the Hilbert-scheme of closed zero dimensional subschemes of length $d$ on the curve $\C$. 
Then it is well-known that  there is an isomorphism 
\begin{equation}
\label{key.isom.1}
\Sym^d(\C) \cong \Hilb^d(\C).
\end{equation}
See, for example, \cite[XVII, Proposition 6.3.9]{SGA4}.
(In fact, it is shown there that there is an isomorphism between the corresponding functors represented by the above schemes.) In view of the above
isomorphism, one may identify points of $\Sym^d(\C)$ with effective zero dimensional cycles of degree $d$ on $\C$.
\vskip .2cm
Clearly there is a natural map $\Hilb^d(\C) \ra Pic^d(\C)$ sending a divisor of degree $d$ to its associated line bundle.
Making use of the isomorphism in ~\eqref{key.isom.1}, this map therefore defines the Abel-Jacbi map of degree $d$:
\begin{equation}
\xi_{\rmd}: \rm{Sym}^d({\C}) \rightarrow \rm{Pic}^d({\C}),\,\,\,\,\,\,\,D \mapsto \cO_{\C}(D).
\end{equation}  
\vskip .2cm
Given a point $Q$ in $\Q(r, d)$, we have the short-exact sequence:
\[ 0 \ra {\mathcal F}(Q) \ra \O_{{\C}}^{\oplus r} \ra Q \ra 0.\]
Sending $Q$ to the scheme-theoretic support of the quotient for the induced homomorphism 
$\wedge ^r {\mathcal F}(Q) \ra \wedge ^r (\O_{\C}^{\oplus r})$, we obtain a 
morphism
\begin{equation}\label{phi}
\phi_{\rmd}: \Q(r, d) \ra \Hilb^d(\C) \cong \Sym^d(C).
\end{equation}
\begin{remark}
When the base field is algebraically closed, the above maps $\phi_{\rmd}$ and $\xi_{\rmd}$ may be defined more directly without making use of
the Hilbert-scheme, since now it suffices to define these maps on
 the corresponding closed points. When the base field is no longer algebraically closed, the use of the isomorphism in ~\eqref{key.isom.1}
 facilitates the definition of the maps $\phi_{\rmd}$ and $\xi_{\rmd}$. (We thank the referee for pointing out this approach to us.)
 \end{remark}
 \vskip .2cm
Throughout the paper,  $\Br'(\rmX)$ will denote the cohomological Brauer group associated to the scheme $\rmX$. Recall this
is defined as the torsion part of the \'etale cohomology group $\H^2_{et}(\rmX, {\mathbb G}_m)$ in general. 
Then the first result of the paper is the following
 theorem.
\begin{theorem}
 \label{main.thm.1}
 If the base field $k$ is  separably closed of characteristic $p\ge 0$, and $\ell$ is any prime invertible in $k$, 
the induced maps $\phi_{\rmd}^*: \Br'(\Sym^d({\C}))_{\ell^n} \ra \Br'(\Q(r,d))_{\ell^n}$ and $\xi_{\rmd}^*:\Br'(\Pic^d({\C}))_{\ell^n} \ra \Br'(\Sym^d({\C}))_{\ell^n}$
are isomorphisms for any $n>0$ and for any $d \ge 3$. Here the subscripts $\ell^n$ denote the $\ell^n$-torsion sub-modules. 
\end{theorem}
 Over the field of complex numbers, it was shown in \cite{BDH} that the maps $\phi_{\rmd}^*$ and $\xi_{\rmd}^*$ are isomorphisms on
the  cohomological Brauer groups, for all $d \ge 2$.  Our present proof
holds in arbitrary characteristics, since we make use of \'etale  cohomology methods to approach the cohomological Brauer groups.
However, the restriction that $d \ge 3$ is needed in our proof, so that we can invoke the weak Lefschetz Theorem: see Propositions
~\ref{weak.l} and ~\ref{LHT}(ii).
\vskip .2cm
The extension to positive characteristics is still far from automatic. One may see from our proof that one needs
to have a calculation of the Picard groups of symmetric powers of projective smooth curves yielding results similar
to that of MacDonald (see \cite{Mac}) for singular cohomology: fortunately this was already worked out by Collino (see \cite{Coll}) for all Chow-groups. 
This plays
a key role in the proof of the isomorphism for $\xi_{\rmd}^*$. In order to deduce the corresponding isomorphism for
$\phi_{\rmd}^*$ one also needs to have a calculation of the Picard groups of the Quot-schemes we are considering. Fortunately for us,
this had also been worked out in the context of their Chow groups by del Bano: see \cite{dB}. The referee also pointed out that the homotopy invariance 
property of the prime to-the-characteristic torsion of the Brauer group is classical and has origins in \cite[Proposition 7.7]{Aus} and appears also 
in \cite{Ga}, \cite[Proposition 8.6, Remark 5.2]{Ro}. 
\vskip .2cm
At the same time we also obtain similar results for Prym varieties, which is the content of our next main result. 
\vskip .2cm
Let $k$ denote a separably closed field of characteristic different from $2$ and let $f:\tilde{{\C}}\rightarrow {\C}$ denote a degree two Galois covering, between smooth projective curves over $k$. Denote  by $\sigma$ the involution acting on $\tilde{{\C}}$. Then we can write ${\C}= \frac{\tilde{{\C}}}{<\sigma>}$. 
Let $\tilde{g}:= \rm{ genus}(\tilde {\C})$ and $g:= \rm{ genus}({\C})$. Then $f$ induces a 
morphism, $f^*: \J({\C}) \rightarrow \J(\tilde{{\C}})$,
on the Jacobian varieties of $\tilde{{\C}}$ and ${\C}$, given by $l\mapsto f^*(l)$, the pullback of a degree zero line bundle $l$ on ${\C}$. When $f$ is ramified, the morphism $f^*$ is injective (see \cite[Corollary 11.4.4]{BL} or \cite[\S 3]{Mum}.).
When $f$ is unramified, the image is a quotient of $\J({\C})$ by a $2$-torsion line bundle $l$ on ${\C}$ which defines the covering $f$.
In either case the image is an abelian subvariety of the Jacobian $\J(\tilde{{\C}})$.
The abelian variety
$$
\P:= \frac{\J(\tilde{{\C}})}{Image(f^*)}
$$
is called the \textit{Prym variety} associated to the degree two covering $\tilde{{\C}}\rightarrow {\C}$. By identifying $\P$ with the kernel of the endomorphism $\sigma +id$ on $\J(\tilde \C)$ (see \cite[\S 3]{Mum}), we can write 
$$
\J(\tilde{{\C}})= Image(f^*)+ \P
$$ 
and there is a finite isogeny
$$
\J({\C})\times \P \rightarrow \J(\tilde{{\C}}).
$$
Let  $r$ denote a fixed integer $\ge 2\tilde g-1$. We will assume henceforth that
$f$ is ramified and that $z$ is a ramification point on $\C$ which is also a $k$-rational point.
Consider the Abel-Jacobi map, from \S \ref{Poincare}:
$\Phi:\Sym^r(\tilde{{\C}})=\mathbb{P}(E_r) \rightarrow \J (\tilde{{\C}})$. Denote the inverse image 
$\rmW^r_P:=\Phi^{-1}(\P)$
and the subvarieties
$\rmW^s_{\P}:= \rmW^r_{\P} \cap (\Sym^s(\tilde{{\C}})+ (r-s).z)$
for $1\leq s\leq r$. 

\begin{theorem}
 \label{main.thm.2}
 (i) The induced pullback map
$$
\Br'(\P) \rightarrow \Br'(W^r_{\P})
$$
is an isomorphism.
\vskip .2cm
(ii) If $s$ is the largest integer $\le r$, so that $\rmW^{s-1}_P \ne \rmW^r_P$, and $dim (\rmW^s_{\P}) \ge 4$, then $\Br(\rmW^r_{\P})_{\ell^n} =\Br(\rmW^s_P)_{\ell^n} \ra \Br(\rmW^{s-1}_P)_{\ell^n}$ is surjective
as long as $\ell$ is a prime different from the characteristic $p$ of the base field. Moreover, this map will be an isomorphism
if $\rmW^{s-1}_P$ is smooth and the induced map $\Pic(\rmW^s_P)/\ell^n \ra \Pic(\rmW^{s-1}_P)/\ell^n $ is also an isomorphism; the last condition is
satisfied if the base field is the complex numbers, $dim(\rmW^s_{\P})\geq 4$ and ${\rm W}^{s-1}_{\P}$ is smooth.
\end{theorem}
\begin{remark} In the first statement there is no need to consider the $\ell^n$-torsion part of the Brauer group for $\ell$ different from
 the characteristic of the base field. This is discussed in the Remark following Gabber's Theorem: see \cite[Theorem 2 and Remark, p.193]{Ga}.
\end{remark}

\vskip .2cm \noindent
{\bf Terminology}. 
Since we need to consider both the Picard scheme and the Picard group, we will use
$\Pic$ ($\bPic$) to denote the Picard scheme (the Picard group, \res). $\H^*_{et}$ will always denote cohomology 
computed on the \'etale site.
\vskip .2cm \noindent
{\bf Acknowledgment}. The second author thanks the Institute for Mathematical Sciences, Chennai, for supporting his visit during the
summer of 2017 and for providing a pleasant working environment during his visit. 
Both the authors thank the referee for undertaking a careful reading of the paper, for 
pointing out references to the homotopy invariance of Brauer groups and for various other 
remarks that have helped the authors improve the exposition as well as some of the results.

\section{Review of Brauer groups of smooth projective varieties over a field $k$}
In this section, we review Brauer groups both from the point of view of complex geometry as well as from the point of view of \'etale
cohomology applicable in positive characteristics. We begin by considering the case where $k = \Cl$.

\subsection{Assume $k=\mathbb{C}$}

Let $\rmX$ denote a smooth projective variety over $\mathbb{C}$. Denote by $\cO_{\rmX}$ and $\cO_{\rmX}^*$, the sheaf of holomorphic functions on ${\rmX}$, and the multiplicative sheaf of nowhere vanishing holomorphic functions on ${\rmX}$.
Consider the exponential sequence:
\begin{equation}
 \label{exp.seq.0}
0\rightarrow \mathbb{Z} \rightarrow \cO_{\rmX} \sta{exp}{\rightarrow} \cO_{\rmX}^* \rightarrow 0.
\end{equation}
The associated long exact sequence of cohomology groups (computed on ${\rmX}({\mathbb C})$ with the complex topology) is the following:
 \begin{equation}
\label{exp.seq.1}
\rightarrow \H^1({\rmX},\cO_{\rmX}^*) \rightarrow \H^2({\rmX},\mathbb{Z}) \rightarrow \H^2({\rmX},\cO_{\rmX})\rightarrow \H^2({\rmX},\cO_{\rmX}^*) \rightarrow \H^3({\rmX},\mathbb{Z})\rightarrow.
\end{equation}
Under the connecting map, the image of $\bPic({\rmX})=H^1({\rmX},\cO_{\rmX}^*)$ inside $H^2({\rmX},\mathbb{Z})$ is the N\'eron-Severi group 
$$
\NS({\rmX}):= \H^{1,1}({\rmX}) \cap \H^2({\rmX},\mathbb{Z}).
$$
\begin{definition} (Complex analytic version)
 \label{coh.Br.grp.0}
The \textit{cohomological Brauer group} $Br'({\rmX})$ is the torsion subgroup of the cohomology group $H^2({\rmX},\cO_{\rmX}^*)$. In other words,
\vskip .2cm
$\Br'({\rmX})\,=\,\H^2({\rmX},\cO_{\rmX}^*)_{tors}$.
\end{definition}

The Brauer group can also be expressed via the  cohomology sequence in ~\eqref{exp.seq.1} as follows.
\begin{proposition}\label{Schroer}
There is a natural short exact sequence:
$$
0\rightarrow \frac{\H^2({\rmX},\mathbb{Z})}{\NS({\rmX})}\otimes \mathbb{Q}/\Z\rightarrow \Br'({\rmX}) \rightarrow \H^3(X,\mathbb{Z})_{tors}\rightarrow 0.
$$
\end{proposition} 
\begin{proof}
See \cite[p.878, Proposition 1.1]{Sc}.
\end{proof}

\subsection{Assume $k$ is any separably closed field} 

Let ${\rmX}$ denote a smooth projective variety over $k$. Let ${\mathbb G}_m$ denote the \'etale sheaf of units in the 
structure sheaf and for each prime $\ell$, let $\mu_{\ell^n}(1)$ denote the sub-sheaf of ${\mathbb G}_m$ consisting of 
the $\ell^n$-roots of unity. Then the analogue of the exponential sequence ~\eqref{exp.seq.0} is the Kummer-sequence:
\begin{equation}
 \label{Kummer.seq.0}
1 \ra \mu_{\ell^n}(1) \ra {\mathbb G}_m {\overset {\ell^n} \ra} {\mathbb G}_m \ra 1
\end{equation}
which holds on the \'etale site ${\rmX}_{et}$ of ${\rmX}$, whenever $\ell$ is invertible in $k$. (See \cite[section 3]{Gr}.)
It is observed (see \cite[p. 66]{Mi}, for example)
that the same sequence is exact on the fppf site (i.e. the flat site) of ${\rmX}$ without the requirement that $\ell$ be invertible in $k$. 
However, since we will need affine spaces to be acyclic in various places in our proofs, we will henceforth restrict to the \'etale site
and assume $\ell$ is invertible in $k$. (For the same reason, we are also forced to restrict to separably closed fields.)
Then we obtain the corresponding long-exact sequence:
\begin{equation}
 \label{Kummer.seq.1}
\rightarrow \H^1_{et}({\rmX},{\mathbb G}_m) {\overset {\ell^n} \ra } \H^1_{et}({\rmX}, {\mathbb G}_m) \ra  \H^2_{et}({\rmX},\mu_{\ell^n}(1)) \rightarrow \H^2_{et}({\rmX}, {\mathbb G}_m) \rightarrow H^2_{et}({\rmX}, {\mathbb G}_m) \rightarrow \cdots 
\end{equation}
which holds on the \'etale site when $\ell$ is invertible in $k$. (Recall $\bPic({\rmX}) = \H^1_{et}({\rmX}, {\mathbb G}_m)$.)
\begin{definition} (Algebraic version)
 \label{coh.Br.grp.1}
The \textit{cohomological Brauer group} $\Br'({\rmX})$ is the torsion subgroup of the cohomology group $\H^2_{et}({\rmX}, {\mathbb G}_m)$. In other words,
 $\Br'({\rmX})\,=\,\H^2_{et}({\rmX}, {\mathbb G}_m)_{tors}$.
\end{definition}
Then one also obtains the short-exact sequences:
\begin{equation}
 \label{Kummer.seq.2}
0 \ra \bPic({\rmX})/\ell^n\cong \NS({\rmX})/\ell^n \ra \H^2_{et}({\rmX}, \mu_{\ell^n}(1)) \ra \Br'({\rmX})_{\ell^n} \ra 0 \,  \mbox{ and}
\end{equation}
\begin{equation}
\label{Kummer.seq.3}
0 \ra \NS({\rmX}) \otimes {\mathbb Z}_{\ell} \ra \H^2_{et}({\rmX}, {\mathbb Z}_{\ell}(1)) \ra T_{\ell}({\Br'}({\rmX})) \ra 0 \,
\end{equation}

\vskip .1cm \noindent
where 
\vskip .1cm
$\NS({\rmX}) = \bPic({\rmX})/\bPic^o({\rmX}) \cong \H^1_{et}({\rmX}, {\mathbb G}_m)/\H^1_{et}({\rmX}, {\mathbb G}_m)^o,$
\vskip .1cm
$\bPic({\rmX})/\ell^n = coker (\bPic({\rmX}) {\overset {\ell^n} \ra} \bPic({\rmX})), \, \NS({\rmX})/\ell^n = coker(\NS({\rmX}) {\overset {\ell^n} \ra} \NS({\rmX})),$
\vskip .1cm
$\Br'({\rmX})_{\ell^n}$ = the $\ell^n$-torsion part of $\Br'({\rmX})$ and $T_{\ell}({\Br'}({\rmX})) = \limn \Br'({\rmX})_{\ell^n}$.
\vskip .2cm
The main tool we use to obtain results on the cohomological Brauer group will be the short-exact sequence in Proposition ~\ref{Schroer} in
the complex analytic framework and the short-exact sequence in ~\eqref{Kummer.seq.2} in the algebraic framework.


\section{Symmetric product of curves}
\label{tech.res}

Suppose ${\C}$ is a smooth geometrically connected projective curve over a field $k$, of genus $g$.

The $\rmd$-self product of the curve ${\C}$ is denoted by ${\C}^{\times \rmd}$ and the $\rmd$-symmetric product (or equivalently the $\rmd$-fold symmetric power) of ${\C}$ is denoted by $\rm{Sym}^d({\C})$.
Then we have the relation:
$$
\rm{Sym}^d({\C})\,=\,\frac{{\C}^{\times d}}{\Sigma_d}.
$$
Here $\Sigma_d$ denotes the symmetric group on $\rmd$-letters.

The variety $\rm{Sym}^d({\C})$ is a smooth geometrically irreducible projective variety of dimension $\rmd$. 

The Jacobian variety of ${\C}$ is denoted by $\J({\C})$. This variety is  the Picard variety $\rm{Pic}^0({\C})$, parameterizing line bundles of degree zero on ${\C}$.
More generally we denote by $\rm{Pic}^d({\C})$, the Picard variety parameterizing line bundles of degree $\rmd$ on ${\C}$.

 \vskip .2cm
Recall that we have assumed that the curve $\C$ has a $k$-rational point ${\rm p}_0$. 
The given $k$-rational point $\rmp_0$ defines the closed immersion
\begin{equation}
\label{i}
i:\rm{Sym}^{d}({\C})\rightarrow \rm{Sym}^{d+1}({\C}),\,\,\,D\mapsto D+\p_0.
\end{equation}
\vskip .2cm \noindent
and the image is a smooth divisor on $\rm{Sym}^{d+1}({\C})$.
 
\subsection{Poincar\'e line bundle and the Symmetric product} (See \cite[Chapter 4]{ACGH} for the case the base field is $\Cl$, and 
\cite{Matt-2}, \cite{Scb} for the case the base field is any algebraically closed field.) Recall that the curve $\C$ is assumed to have a
 $k$-rational point.
 
\label{Poincare}

The Poincar\'e line bundle $\cL \rightarrow \, {\C} \times \Pic^r({\C})$ parametrizes 
line bundles of degree $r$, i.e. when restricted to ${\C}\times [l]$, it is the line bundle $l$ on ${\C}$, and satisfying a universal property.

Denote the projection $\pi: {\C}\times \Pic^r({\C}) \rightarrow \Pic^r({\C})$. The push-forward of $\cL$ is denoted by 
$$
E_r:=\pi_*\cL
$$
and its projectivization by 
$$
\Phi:\mathbb{P}(E_r)\rightarrow \Pic^r({\C}).
$$
When $r\geq 2g-1$, $E_r$ is a vector bundle and $\mathbb{P}(E_r)\rightarrow \Pic^r({\C})$ is a projective bundle. (We skip the verification of
these assertions for the case the base field is only separably closed, as it may be readily deduced from the case when the base field is algebraically closed.)

Then the variety $\rm{Sym}^r({\C})\simeq \mathbb{P}(E_r)$ and the morphism $\Phi$ is identified with the Abel-Jacobi map, considered in the previous subsection.
The fiber $\Phi^{-1}(\cL)$ is the complete linear system $|\cL|$ of a point $\cL\in \Pic^r({\C})$. 

\subsection{Weak Lefschetz for the Symmetric product}

As observed in ~\eqref{i}, the chosen $k$-rational point $\p_0$ defines the closed immersion
$i:\rm{Sym}^{d}({\C})\rightarrow \rm{Sym}^{d+1}({\C}),\,\,\,D\mapsto D+ \p_0.$
 Moreover, we have the following.
\begin{lemma}\label{amplediv}
When $k$ is the field of complex numbers, the image of $\rm{Sym}^{d}({\C})$ is an ample smooth divisor on $\rm{Sym}^{d+1}({\C})$. The same 
conclusion holds when $k$ is any field of characteristic $0$.
\end{lemma}
\begin{proof}
See \cite[p.310, Proof of (2.2) Proposition]{ACGH} for the case the field $k = \Cl$. Recall the point $\p_0$ is a $k$-rational point by our assumptions.
 Therefore, ampleness of the divisor descends along extension of scalars to prove the second statement. (We thank the referee for this
observation.)
\end{proof}
\vskip .2cm
Next assume $k$ has arbitrary characteristic $p$. Let ${\rmX}= \Sym^{d+1}({\C})$ and ${\rmY}= \Sym^d({\C})$, for a smooth projective curve ${\C}$. Then $dim({\rmX}) =d+1$ and $dim({\rmY}) =d$
with ${\rmX}-{\rmY} = \Sym^{d+1}({\C}-\{\rmp_0\})$ where ${\rmY}$ is imbedded in ${\rmX}$ by the map $\rmD \mapsto D+\rmp_0$ in ~\eqref{i}. Clearly $\C -\{\rmp_0\}$ is affine (see,
for example, \cite[Chapter 11, Proposition 4.1]{Hart}), and
therefore so is $\Sym^{d+1}(\C-\{\rmp_0\})$. 
\begin{proposition} (See the proof of \cite[9.4 Corollary]{FK}.)
  \label{weak.l}
Suppose ${\rmY} \ra {\rmX}$ is a closed immersion of  projective smooth connected schemes over an algebraically closed field, so that the
complement $U={\rmX}-{\rmY}$ is affine and $\ell$ is a prime different from $char(k)$. Assume that $dim(\rmX) =\d$ and that $dim({\rmY}) = \d-1$. Then in the long-exact sequence in \'etale cohomology with proper supports
\[ \cdots \ra \H^{q-1}_{et}({\rmY}, \mu_{\ell^n}(\t)) \ra \H^q_{et,c}(U, \mu_{\ell^n}(\t)) \ra \H^q_{et}({\rmX}, \mu_{\ell^n}(\t)) \ra \H^q_{et}({\rmY}, \mu_{\ell^n}(\t))\]
\[ \ra \H^{q+1}_{et, c}(U, \mu_{\ell^n}(\t)) \ra \cdots, 
\]
$\H^q_{et, c}(U, \mu_{\ell^n}(\t)) \cong \H^{q+1}_{et, c}(U, \mu_{\ell^n}(\t)) \cong 0$ for all $q\le \d-2$, so that the
restriction map 
\[\H^q_{et}({\rmX}, \mu_{\ell^n}(\t)) \ra \H^q_{et}({\rmY}, \mu_{\ell^n}(\t))\]
 is an isomorphism for all $\i$ and for all $q \le \d-2$, i.e. for all
$q \le dim({\rmY})-1$. A corresponding results holds with integral (singular) cohomology with proper supports if the schemes
are defined over the field of complex numbers. The same conclusions hold as long as $\rm X$ is smooth and $\rmU$ is affine, even if $\rmY$ is not necessarily smooth.
\vskip .2cm

\end{proposition}
\begin{proof}
 By Poincar\'e duality, we obtain the isomorphism $\H^q_{et,c}(U, \mu_{\ell^n}(\t)) \cong \H^{2d-q}_{et}(U, \mu_{\ell^n}(\d-\t))^{\vee}$ which vanishes 
for $2\d-q>\d$ since the cohomological dimension of the affine scheme ${\rm U}$ with respect to the locally constant sheaf $\mu_{\ell^n}(\d-\t)$ is $\d$. Here $\mu_{\ell^n}(\d-\t) = \Hom(\mu_{\ell^n}(\t), {\mathbb Z}/\ell^n)(\d)$ 
(where the $\Hom$ denotes the sheaf-Hom)  and $\H^{2d-q}_{et}(U, \mu_{\ell^n}(d-t))^{\vee}$ denotes the dual 
\newline \noindent
${\rm {Hom}}(\H^{2d-q}_{et}(U, \mu_{\ell^n}(\d-\t)), {\mathbb Z}/\ell^n)$.
The condition $2\d-q>\d$ is clearly equivalent to $q<d$. The long-exact sequence in \'etale cohomology is obtained from the 
short exact sequence: $1 \ra j_!j^*(\mu_{\ell^n}(\t)) \ra \mu_{\ell^n}(\t) \ra i_*i^*(\mu_{\ell^n}(\t)) \ra 1$, where $j:\rmU \ra \rmX$ and $i:\rmY \ra \rmX$ are 
the given immersions. Therefore, the long exact sequence exists even if $Y$ is singular. Observe also that the smoothness of $\rmX$ implies that of $\rmU$.
The above arguments show that all that is needed for the conclusions to hold is for $\rmX$ to be smooth and $\rmU$ to be  affine.
\end{proof}
\begin{remark}
 The above Proposition shows that we obtain the
{ weak Lefschetz} isomorphism without checking that the divisor ${\rmY}$ is ample, but by just knowing that ${\rmX}-{\rmY}$ is affine and smooth. 
A variant of the above Proposition also appears in \cite[Expose {\rmX}IV, Corollaire 3.3]{SGA4}, but it is stated in a somewhat different form, involving local cohomology.
\end{remark}

\vskip .2cm \noindent
We conclude this section with the following weak Lefschetz theorem.
\begin{proposition}
\label{LHT}
(i) Assume that the curve $\C$ is defined over the complex numbers. Then the restriction map
\[i^*:\H^2(\rm{Sym}^{d+1}({\C}),\Z) \rightarrow \H^2(\rm{Sym}^{d}({\C}),\Z)\]
is an isomorphism, for each $d\geq 3$. Moreover, this is an isomorphism of Hodge structures.
\vskip .2cm
(ii) Assume that the curve $\C$ is defined over an algebraically closed field $k$ of characteristic $p$ and let $\ell$ denote a
prime $\ne p$. Then the restriction map
\[i^*:\H^2_{et}(\rm{Sym}^{d+1}({\C}),\mu_{\ell^n}) \rightarrow \H^2_{et}(\rm{Sym}^{d}({\C}),\mu_{\ell^n})\]
is an isomorphism, for each $d \geq 3$. The same results holds if the curve $\C$ is only defined over a separably closed field $k$, but has
a $k$-rational point. 
\end{proposition}
\begin{proof}
W consider (i) first. By Lemma \ref{amplediv}, the divisor $\rm{Sym}^{d}({{\C}})+\rmp_0$ is a 
smooth ample divisor on  $\rm{Sym}^{d+1}({{\C}})$. When $d-1\geq 2$, the restriction map on cohomology groups is an isomorphism of 
Hodge structures, by the weak Lefschetz theorem.
\vskip .2cm
Next we consider (ii). The case when the base field is algebraically closed follows from Proposition ~\ref{weak.l} by taking ${\rmX}= \rm{Sym}^{d+1}({\C})$ and ${\rmY}= \rm{Sym}^{d}({\C})$
and observing that the dimension of $\rm{Sym}^{d+1}({\C}) = d+1$. The extension to the case the base field is only separably closed holds,
since $i: \Sym^d(\C) \ra \Sym^{d+1}(\C)$ is a closed immersion and one may then make use of Lemma ~\ref{rad.ext} below to reduce to the 
case the base field is algebraically closed.
\end{proof}

\begin{remark}
A proof of the first statement in the above proposition is given in \cite{BDH}, by finding explicit generators of the cohomology groups
$\H^2(\rm{Sym}^d({{\C}}),\Z)$. By using the weak Lefschetz theorem, we are able to avoid this step altogether, though
with a slight penalty that we need to restrict to the case where the degree $d$ is always at least 3, whereas in \cite{BDH},
it is possible to allow $d$ to be also $2$.
\end{remark}

\section{Brauer groups of $\rm{Sym}^d({\C})$ and $\Pic^d({\C})$ in positive characteristics}
We begin with the following result that enables one to pass from separably closed base fields to algebraically closed base fields.
\begin{lemma}
 \label{rad.ext}
Let $k\subseteq k'$ denote a radicial (i.e purely inseparable) extension of fields and that $char(k)=p>0$. If ${\rmX}$ is a scheme of finite type over $k$ and ${\rmX}'$ is its base extension to $k'$,
then the induced map $f:{\rmX}' \ra {\rmX}$ is a finite and flat map of degree some power of $p$. Therefore, it induces an isomorphism 
$f^*: CH^*({\rmX}) \otimes_{\mathbb Z} {\mathbb Z}/\ell^n \ra CH^*({\rmX}') \otimes_{\mathbb Z} {\mathbb Z}/\ell^n$ for any prime $ \ell \ne p$ and
any $n>0$. Moreover $f^*: \H_{et}^*({\rmX}, \mu_{\ell^n}) \ra \H^*_{et}({\rmX}', \mu_{\ell^n})$ is also an isomorphism.
\end{lemma}
\begin{proof} Clearly the induced map $\pi:Spec \, k' \ra Spec \, k$ is a finite and flat map of degree some power of $p$. Then one has
 an induced push-forward map $f_*: CH^*({\rmX}', {\mathbb Z}/\ell^n) \cong CH^*({\rmX}') \otimes_{\mathbb Z} {\mathbb Z}/\ell^n \ra
CH^*({\rmX}, {\mathbb Z}/\ell^n) \cong CH^*({\rmX}') \otimes_{\mathbb Z} {\mathbb Z}/\ell^n$ so that the compositions $f_*\circ f^* $, $f^* \circ f_*$  are multiplications
 by the degree of $f$, which is a power of $p$. Here the statement that $f_*\circ f^*$ is multiplication by the degree of $f$ is a consequence
 of the projection formula and does not use the fact that $f$ is radicial. One needs the assumption that $f$ is radicial to conclude
 that the composition  $f^*\circ f_*$ is also multiplication by the degree of $f$. (A proof of this statement in the
 context of Algebraic K-theory may be found in \cite[\S 7, Proposition 4.8]{Qu} and may be adapted to Chow groups and \'etale cohomology 
 readily.)
 \vskip .2cm
 But any power of $p$ is a unit in ${\mathbb Z}/\ell^n$ so that the map $f^*$ is an isomorphism.
This proves the first statement. The second statement follows similarly.
\end{proof}

We now obtain the following theorem.

\begin{theorem}
 \label{Sym.and.Pic}
Assume that ${\C}$ is a smooth projective curve over the separably closed field $k$.
\vskip .1cm
(i) Then the restriction map $Br'(\Sym ^{d+1}({\C}))_{\ell^n} \ra Br'(\Sym^{d}({\C}))_{\ell^n}$ is an isomorphism for $d \ge 3$.
\vskip .2cm
(ii) The Abel-Jacobi maps $\xi_{\rmd}:\Sym^d({\C}) \ra Pic^d({\C})$ induce isomorphisms $\xi_{\rmd}^*:\Br'(\Pic^d({\C}))_{\ell^n}  \ra Br'(\Sym^d({\C}))_{\ell^n}$  for all $d \ge 3$.
\end{theorem}
\begin{proof} 
 For (i) we now consider the commutative diagram:
\[\xymatrix{{0} \ar@<1ex>[r] & {\bPic(\Sym^{d+1}({\C})))/\ell^n } \ar@<1ex>[r] \ar@<1ex>[d] & {\H^2_{et}(\Sym^{d+1}({\C}), \mu_{\ell^n})} \ar@<1ex>[r] \ar@<1ex>[d] & { \Br'(\Sym^{d+1}({\C}))_{\ell^n}} \ar@<1ex>[r] \ar@<1ex>[d]  & {0}\\
 {0} \ar@<1ex>[r] & {\bPic(\Sym^{d}({\C})))/\ell^n } \ar@<1ex>[r]  & {\H^2_{et}(\Sym^{d}({\C}), \mu_{\ell^n})} \ar@<1ex>[r]  & { \Br'(\Sym^{d}({\C}))_{\ell^n}} \ar@<1ex>[r]   & {0}.}\\
 \]
By Lemma ~\ref{rad.ext}, one observes that replacing the base field, which is assumed to be separably closed, by its algebraic closure,
induces an isomorphism on the terms appearing in the first two columns, so that the one obtains an induced isomorphism on
the $\ell^n$-torsion parts of the cohomological Brauer groups. Therefore, one may assume the base field is algebraically closed.
By \cite[Theorem 2]{Coll}, the left-most map is surjective. By the weak Lefschetz theorem, the middle vertical map is an isomorphism
for $d-1\ge 2$, i.e. for $d \ge 3$, since the dimension of $\Sym^{d}({\C}) = d$. Since the diagram clearly commutes,
the left-most map is also an injection, and therefore it is also an isomorphism. Therefore, the five Lemma shows that the
last vertical map is also an isomorphism for $ d \ge 3$. This proves (i).
\vskip .2cm
For (ii) we consider the commutative square:
\[\xymatrix{{\Br'(\Pic^{d+1}({\C}))_{\ell^n} } \ar@<1ex>[r] \ar@<1ex>[d]^{\xi_{d+1}^*} & {\Br'(\Pic^{d}({\C}))_{\ell^n}} \ar@<1ex>[d]^{\xi_{\rmd}^*} \\
 {\Br'(\Sym^{d+1}({\C}))_{\ell^n} } \ar@<1ex>[r]  & {\Br'(\Sym^{d}({\C}))_{\ell^n}} .}
 \]
Since the top row is evidently an isomorphism, and the bottom row is an isomorphism for $d \ge 3$, it suffices to show
that the map $\xi_{\rmd}^*:\Br'(\Pic^d({\C}))_{\ell^n} \ra \Br'(\Sym^d({\C}))_{\ell^n}$ is an isomorphism for all $\rmd$ large. For $d \ge 2g-1$, we already
observed in \S ~\ref{Poincare} that $\xi_{\rmd}:\Sym^d({\C}) \ra \Pic^d({\C})$ is a projective space bundle associated to a vector bundle. Therefore,
by \cite[p. 193]{Ga}, the map $\Br'(\Pic^d({\C}))_{\ell^n} \ra \Br'(\Sym^d({\C}))_{\ell^n}$ is an isomorphism. (One may also
 note that \cite{Aus} and  \cite{El} are results related to \cite[p. 193]{Ga}.) This completes the proof
of (ii) and hence that of the theorem.

\end{proof}
\begin{remark} Clearly the above theorem proves the second isomorphism in Theorem ~\ref{main.thm.1}.
 \end{remark}

\section{Motivic cohomology, \'Etale cohomology and Brauer groups of Quot-schemes associated to projective smooth curves}

We begin by briefly reviewing Motivic Thom classes, which will be needed in the proofs in this section.

\subsection{Motivic Thom classes} 
A theory of Thom-classes and Chern classes for vector bundles on algebraic varieties has been discussed in detail
by Panin in the motivic context: see \cite[\S 1]{Pan}. We summarize the main results here. One may begin with the
homotopy purity theorem (see \cite[Theorem  2.21]{MV}) which says the following: Given ${\rmY} \ra {\rmX}$ a closed immersion of
smooth schemes of pure dimension over a field $k$, and with $\rmN$ the normal bundle associated to the above immersion, one obtains
an isomorphism of pointed (simplicial) sheaves ${\rmX}/({\rmX}-{\rmY}) \simeq \rmN/(\rmN-{\rmY})$ in the motivic homotopy category over $k$, i.e.
the homotopy category of pointed simplicial sheaves on the Nisnevich site of $k$ and with ${\mathbb A}^1$  inverted.
Let the above homotopy category be denoted $HSp_k$. This is the unstable motivic homotopy category.
\vskip .2cm

Let $\Sm/k$ denote the category whose objects are pairs $({\rmX}, \rmU)$ where ${\rmX}$ is a smooth scheme over $k$ of finite type and
$U$ open in ${\rmX}$. Morphisms $({\rmX}', \rmU') \ra ({\rmX}, \rmU)$ are maps of schemes $f: {\rmX}' \ra {\rmX}$ so  that $f(\rmU') \subseteq \rmU$.
A cohomology theory on $\Sm/k$ (see \cite[Definition 1.1]{Pan}) is a contravariant functor 
\[\rmA^{\bullet, *}:\Sm/k \ra \mbox{(bi-graded abelian groups)}\]
so that it has localization sequences relating the cohomology of ${\rmX} =({\rmX}, \phi)$, $\rmU=(\rmU, \phi)$ and $({\rmX}, \rmU)$ in the usual manner,
it is homotopy invariant (in the sense that $\rmA^{\bullet, *}({\rmX} \times {\mathbb A}^1) \simeq \rmA^{\bullet, *}({\rmX})$) and has \'etale excision, i.e. if $e: ({\rmX}', \rmU')  \ra ({\rmX}, \rmU)$ is
\'etale and the induced map $ {\rmX}' -\rmU' \ra {\rmX}-\rmU$ is an isomorphism, then the induced map $e^*:\rmA^{\bullet, *}({\rmX}, \rmU) \ra \rmA^{\bullet, *}({\rmX}', \rmU')$
is an isomorphism. Such a cohomology theory is multiplicative, if there is a bi-graded pairing
\[\rmA^{p', q'}({\rmX}', \rmU') \otimes \rmA^{p, q}({\rmX}, \rmU) \ra 
\rmA^{p'+p, q'+q}({\rmX}\times {\rmX}', {\rmX}'\times \rmU \cup \rmU' \times {\rmX}))\]
that is associative in the obvious sense, there is a unique 
class $1 \in \rmA^{0,0}(Spec \, k)$ that acts as the identity for the above pairing and the 
above pairing
satisfies a partial Leibniz-rule: see \cite[Definition 1.5]{Pan}.

\vskip .2cm
We will also assume that the given cohomology theory extends uniquely to a contravariant functor 
\[HSp_k \ra \mbox{(bi-graded abelian groups)}\]
so that $\rmA^{\bullet, *}({\rmX}, \rmU) $ identifies with $\rmA^{\bullet, *}({\rmX}/\rmU)$.
It follows in particular,
that one obtains an isomorphism $\rmA^{\bullet, *}({\rmX}, \rmU) \simeq \rmA^{\bullet, *}(\rmN, \rmN -({\rmX}-\rmU))$. (One may often apply a deformation to the normal cone to obtain
this isomorphism explicitly, depending on the cohomology theory.) If we denote ${\rmX}- \rmU$ by the closed subscheme ${\rmY}$, and denote
$\rmA^{\bullet, *}({\rmX}, \rmU)$ ($\rmA^{\bullet, *}(\rmN, \rmN-{\rmY})$) by $\rmA_{\rmY}^{\bullet, *}({\rmX})$ ($\rmA_{\rmY}^{\bullet, *}(\rmN)$, \res), the above isomorphism may be denoted $\rmA_{\rmY}^{\bullet, *}({\rmX}) \simeq \rmA_{\rmY}^{\bullet, *}(\rmN)$.
\vskip .2cm
Assume $\rmA^{\bullet, *}$ is a multiplicative cohomology theory in the above sense. Then this cohomology theory has {\it Thom classes}
if for every smooth scheme ${\rmX}$ and every vector bundle $\rmE$ of rank $c$ over ${\rmX}$, there is a unique class 
${\rm Th}(\rmE) \in \rmA^{2c, c}_{\rmX}(\rmE)$, so that taking cup product with the class ${\rm Th}(\rmE)$ induces an isomorphism (which we call {\it Thom isomorphism}) :
\[\rmA^{\bullet, *}({\rmX}) \ra \rmA^{\bullet+2c, *+c}_{\rmX}(\rmE).\]
The above isomorphism is further required to be functorial
with respect to maps $f: {\rmX}' \ra {\rmX}$ of smooth schemes. In this situation, one obtains a {\it Gysin-map} associated to each closed immersion 
$i:{\rmY} \ra {\rmX}$ of smooth schemes of pure codimension $c$ by defining
\begin{equation}
 \label{Gysin}
i_*: \rmA^{p, q}({\rmY}) \ra \rmA^{p+2c, q+c}({\rmX}) 
\end{equation}
as follows. We let $i_*(\alpha) $ be given by the image of  ${\rm Th}(\rmN) \cup  \alpha$ under the map 
$\rmA^{p+2c, q+c}_{\rmY}(\rmN) \cong {\rm A}^{p+2c, q+c}_{\rmY}({\rmX}) \ra {\rm A}^{p+2c, q+c}({\rmX})$.
\vskip .2cm
It is then shown in \cite[(1.2.2)]{Pan} that providing the cohomology theory $\rmA^{\bullet, *}$ with Thom-classes is equivalent to
providing $\rmA^{\bullet, *}$ with a theory of Chern-classes for vector bundles. 
\vskip .2cm
An important example of such a cohomology theory is motivic cohomology: $\H_{\M}^{p, q}({\rmX}) \cong {\rm CH}^{q}({\rmX}, 2q-p)$
where ${\rm CH}^{q}({\rmX}, 2q-p)$ denotes the higher Chow groups of Bloch. Taking $p=2q$, we then obtain the isomorphism
$\H_{\M}^{2q, q}({\rmX}) = {\rm CH}^q({\rmX}, 0) = {\rm CH}^q({\rmX})$. Another example of such a cohomology theory is \'etale cohomology: ${\rmX} \mapsto \H_{et}^{p}({\rmX}, \mu_{\ell^n}(q))$,
with $\ell $ a prime different from $char (k) =p$ and where $\mu_{\ell^n}(q)$ denotes the sheaf $\mu_{\ell^n}$, Tate-twisted $q$-times.
\vskip .2cm
The cycle map from Chow-groups to \'etale cohomology clearly preserves
Chern classes of vector bundles and therefore also preserves Thom-classes of vector bundles. Moreover the cycle map is compatible with
localization sequences and therefore with the Gysin map $i_*$ associated to closed immersions of smooth schemes.

\subsection{Proof of the isomorphism $\phi_{\rmd}^*: \Br'(\Sym^d({\C}))_{\ell^n} \ra \Br'(\Q(r,d))_{\ell^n}$}
We will assume once again that, {\it throughout this section, the base field is separably closed.}
Given  a projective smooth curve ${\C}$ over $k$, we let $\Q(r, d)$ denote the Quot-scheme parameterizing all torsion quotients
of degree $\rmd$  of the $\O_{{\C}}$-module, $\O_{{\C}}^{\oplus r}$. Then $\Q(r, d)$ is a smooth projective variety of dimension $rd$ over $k$. 
Given a point $Q$ in $\Q(r, d)$, we have the short-exact sequence:
\[ 0 \ra {\mathcal F}(Q) \ra \O_{{\C}}^{\oplus r} \ra Q \ra 0.\]
Sending $Q$ to the scheme-theoretic support of the quotient for the induced homomorphism $\wedge ^r {\mathcal F}(Q) \ra \wedge ^r (\O_{\C}^{\oplus r})$, which will
 be a degree $\rmd$ effective zero dimensional cycle on $\C$, we obtain a 
morphism
\begin{equation}
     \label{phi}
\phi_{\rmd}: \Q(r, d) \ra \Hilb^d(\C) \cong \Sym^d({\C}).
    \end{equation}
\vskip .2cm
Next one observes that there is a natural action of ${\mathbb G}_m^r$ on $\Q(r,d)$ induced from the natural action of ${\mathbb G}_m$ on $\C$: see \cite{Bi} and also \cite[\S5]{BDH}. Let ${\mathbf {Part}}^k_r =
\{ {\mathbf m} =(m_1, \cdots , m_r)\}$ denote the set of partitions of $k$ of length $r$, i.e. sequences of integers
$(m_1, \cdots, m_r)$ so that $m_i \ge 0$ and $\Sigma_i m_i = k$. The connected components of the fixed point scheme, for the above action of ${\mathbb G}_m^r$ on
$\Q(r, d)$, are the $\Sym^{{\mathbf m}}({\C}) = \Sym^{m_1}({\C}) \times \cdots \times \Sym^{m_r}(\C)$. One obtains an associated Bialynicki-Birula
 decomposition of $\Q(r,d)$ which are given by $\{\Sym^{{\mathbf m} +}({\C})| {\mathbf m}\}$ where $\Sym^{{\mathbf m}+}({\C})$ is an affine space bundle over the 
corresponding component $\Sym^{{\mathbf m}}({\C})$. The open cell corresponds to ${\mathbf m}_1 = (0, \cdots, 0, \rmd)$ and
the second largest cell corresponds to ${\mathbf m}_2 = (0, \cdots, 0, 1, \rmd-1)$ so that $\Sym^{{\mathbf m}_1}({\C}) = \Sym^d({\C})$ and
$\Sym^{{\mathbf m}_2}({\C}) = \Sym^{d-1}({\C})\times {\C}$. 

Therefore, the discussion in \cite[Theorem 2.4]{dB} applies to show that the
long-exact localization sequences in Chow groups as well as \'etale cohomology with $\mu_{\ell^n}$-coefficients associated to the stratification of $\Q(r,d)$ by $\{\Sym^{{\mathbf m} +}({\C})|{\mathbf m}\}$
break up into short-exact sequences. i.e. the long-exact sequences
 \begin{equation}
   \label{loc.seq.et}
\cdots \ra \H^0_{et}(\Sym^{d-1}({\C})\times {\C}, \mu_{\ell^n}) {\overset {f^{et}_*} \ra} \H^2_{et}(Q(r,d), \mu_{\ell^n}) \ra \H^2_{et}(\Sym^d({\C}), \mu_{\ell^n}) \ra \cdots \mbox{ and}
    \end{equation}
\begin{equation}
 \label{loc.seq.mot}
\cdots \ra \CH^0(\Sym^{d-1}({\C})\times {\C}) {\overset {f^{CH}_*} \ra} \CH^1(Q(r,d)) \ra \CH^1(\Sym^d({\C})) \ra \cdots
\end{equation}
break up into short-exact sequences. Observe that for $\H^2_{et}$ and $\CH^1$, no strata of lower dimension other than $\Sym^{{\mathbf m}_2+}({\C})$-contribute,
so that 
\[\H^2_{et}(\Q(r, d), \mu_{\ell^n}) \cong H^2_{et}(\Sym^{{\mathbf m}_1+}({\C}) \cup \Sym^{{\mathbf m}_2+}({\C}), \mu_{\ell^n}) \mbox{ and }\]
\[\CH^1(\Q(r,d)) \cong \CH^1(\Sym^{{\mathbf m}_1+}({\C}) \cup \Sym^{{\mathbf m}_2+}({\C})). \]

Here $f: \Sym^{{\mathbf m}_2+}({\C}) \ra \Q(r, d)$ is the corresponding closed immersion (which is of codimension $1$), and
$f_*^{\CH}$ and $f^{et}_*$ denote the corresponding Gysin homomorphisms. 
 Therefore, the splitting of the long-exact sequences in ~\eqref{loc.seq.et} and ~\eqref{loc.seq.mot} provides the calculation:
\begin{equation}
 \label{key.calc.quot}
 \CH^1(\Q(r, d)) \cong \CH^1(\Sym^d({\C})) \oplus Im(f^{\CH}_*), \, \H^2_{et}(\Q(r,d), \mu_{\ell^n}) \cong \H^2_{et}(\Sym^d(C), \mu_{\ell^n}) \oplus Im(f^{et}_*).
\end{equation}
Moreover one may observe that, since $\Sym^{d-1}({\C}) \times {\C}$ is connected, ${\rm Im}(f^{\CH}_*) = \{f^{\CH}_*(1)\} ={\mathbb Z}$ and
${\rm Im}(f^{et}_*) = \{f^{et}_*(1)\}={\mathbb Z}/\ell^n$. Observe that $f^{{\rm CH}}_*(1)$ is the image of the Thom-class
\[ {\rm Th}\,  \in \, \CH^1_{\Sym^{{\mathbf m}_2+}({\C})}(\Q(r, d))\]
under the map $\CH^1_{\Sym^{{\mathbf m}_2+}({\C})}(\Q(r,d)) \ra \CH^1(\Q(r, d))$, while $f^{et}_*(1)$ is the image of the corresponding Thom-class in \'etale cohomology. Clearly the Thom-class
in Motivic cohomology maps to the corresponding Thom-class in \'etale cohomology. Moreover the definition of the Gysin map as in ~\eqref{Gysin}
shows they are compatible under the cycle map.
\vskip .2cm
Next we also observe that the map $\phi_{\rmd}: \Q(r,d ) \ra \Sym^{{\mathbf m}_1}({\C}) $ and the open 
inclusion $j:\Sym({\C})^{{\mathbf m}_1+} \ra Q(r, d)$ are so that, at the level of the above cohomology groups, $j^* \circ \phi_{\rmd}^* =id$,
i.e. the induced map $\phi_{\rmd}^*$ is a split monomorphism.
\begin{theorem}
 \label{quot.iso}
The map $\phi_{\rmd}: \Q(r, d) \ra \Sym^d(\C)$ induces an isomorphism on the cohomological Brauer groups $\Br'(\quad ) _{\ell^n}$.
\end{theorem}
\begin{proof}
Making use of  Lemma ~\ref{rad.ext}, we may once again assume the base field is algebraically closed.
 We next consider the commutative diagram where the vertical maps are the maps $\phi_{\rmd}^*$:
\[\xymatrix{{0} \ar@<1ex>[r] & {\bPic(\Sym^d({\C}))/\ell^n } \ar@<1ex>[r] \ar@<1ex>[d] & {\H^2_{et}(\Sym^d({\C}), \mu_{\ell^n})} \ar@<1ex>[r] \ar@<1ex>[d] & { \Br'(\Sym^d({\C}))_{\ell^n}} \ar@<1ex>[r] \ar@<1ex>[d]  & {0}\\
 {0} \ar@<1ex>[r] & {\bPic(\Q(r,d))/\ell^n } \ar@<1ex>[r]  & {\H^2_{et}(\Q(r, d), \mu_{\ell^n})} \ar@<1ex>[r]  & { \Br'(\Q(r,d))_{\ell^n}} \ar@<1ex>[r]   & {0}}\\
 \]
On identifying $\CH^1$ with the Picard-groups, the observations above prove that the two left-most vertical maps are split injections and that their co-kernels are isomorphic.
Therefore, a snake lemma argument readily shows that the last vertical map is also an isomorphism, thereby proving the 
theorem.
\end{proof}
\begin{remark} Clearly the above theorem proves the first isomorphism in Theorem ~\ref{main.thm.1}.
 \end{remark}

\section{Brauer groups of Prym varieties}

We begin the section by recalling Prym varieties and later obtain various results on their Brauer groups. 
We may assume throughout this section that $k$ denotes a separably closed field of characteristic different from $2$, though
the strongest results are possible only for the case $k=\mathbb{C}$ as shown below.

\subsection{Prym varieties} \cite[Chapter 12]{BL} and \cite{Mum}

Suppose $f:\tilde{{\C}}\rightarrow {\C}$ is a degree two Galois covering, between smooth projective  curves defined over a separably closed field. Denote by $\sigma$ the involution  acting on $\tilde{{\C}}$. Then we can write ${\C}= \frac{\tilde{{\C}}}{<\sigma>}$. 

Let $\tilde{g}:= \rm{ genus}(\tilde {\C})$ and $g:= \rm{ genus}({\C})$.
 
There is an induced morphism on the Jacobian varieties of $\tilde{{\C}}$ and ${\C}$:
$$
f^*: \J({\C}) \rightarrow \J(\tilde{{\C}})
$$
given by $l\mapsto f^*(l)$, the pullback of a degree zero line bundle $l$ on ${\C}$.

Recall the \textit{Prym variety} associated to the degree two covering $\tilde{{\C}}\rightarrow {\C}$ is defined by:
$$
\P:= \frac{\J(\tilde{{\C}})}{Image(f^*)}.
$$

By identifying $\P$ with the kernel of the endomorphism $\sigma +id$ on $\J(\tilde \C)$ (see \cite[section 3]{Mum}), we can write 
$$
\J(\tilde{{\C}})= Image(f^*)+ \P,
$$ 
and there is an isogeny
$$
\J({\C})\times \P \rightarrow \J(\tilde{{\C}}).
$$
 In particular, $dim(\P)=\tilde{g}-g$. If $f$ is unramified, then $\tilde{g}=2g-1$ and $dim(\P)=g-1$. If $f$ is ramified, then $\tilde{g}$ and $dim(\P)$ can be computed by the Riemann-Hurwitz formula.

\subsection{Special subvarieties of $\rm{Sym}^r(\tilde{{\C}})$}

For simplicity, we assume that $f$ is \textit{ramified} in the following discussion. Let $r$ denote a fixed integer $\ge 2\tilde g-1$.

Note that the action  by $\sigma$ on $\tilde{{\C}}$ naturally extends onto the symmetric products:
$$
\sigma: \Sym^r(\tilde{{\C}}) \rightarrow \Sym^r(\tilde{{\C}}),\,\,\,\,\,\ {\rm D}\mapsto  \sigma({\rm D}).
$$

Let $z\in \tilde{{\C}}$ be a ramification point of $f$: we will assume that $z$ is {\it a $k$-rational point}. Then $z$ is a fixed point for the action of $\sigma$ on $\tilde{{\C}}$, i.e., $\sigma(z)=z$.

This gives a $\sigma$-equivariant filtration:
\begin{equation}\label{equiv}
\tilde{{\C}}+ (r-1).z \subset Sym^2(\tilde{{\C}})+ (r-2).z \subset ... \subset Sym^{r-1}(\tilde{{\C}})+ z \subset Sym^r(\tilde{{\C}}).
\end{equation}

Consider the Abel-Jacobi map, from \S \ref{Poincare}, with respect to $z$:
$$
\Phi:\Sym^r(\tilde{{\C}})=\mathbb{P}(E_r) \rightarrow \J(\tilde{{\C}}).
$$

Denote the inverse image 
$$
\rmW^r_{\P}:=\Phi^{-1}(\P),
$$
and the subvarieties
$$
\rmW^s_{\P}:= W^r_{\P} \cap (\Sym^s(\tilde{{\C}})+ (r-s).z)
$$
for $1\leq s\leq r$. 
 
\subsection{Proof of Theorem ~\ref{main.thm.2}: Brauer groups of ${\rm W}^s_{\P}$ and $\P$}

\vskip .2cm \noindent

We start by recalling  the following results from Proposition ~\ref{LHT}: the maps induced by the closed immersion 
$i: \Sym^{d}(\tilde \C) \ra \Sym^{d+1}(\tilde \C)$
\[i^*:\H^2(\rm{Sym}^{d+1}({\tilde \C}),\Z) \rightarrow \H^2(\rm{Sym}^{d}({\tilde \C}),\Z) \mbox{ and}\]
\[i^*:\H^2_{et}(\rm{Sym}^{d+1}({\tilde \C}),\mu_{\ell^n}) \rightarrow \H^2_{et}(\rm{Sym}^{d}({\tilde \C}),\mu_{\ell^n})\]
are isomorphisms for $d \ge 3$, where the base field $k=\Cl$ for the first isomorphism and $k$ is any separably closed field,
with $char(k) \ne \ell$ for the second isomorphism.
\vskip .2cm
Since $r$ is assumed to be at least $\tilde g-1$, one may observe that
$$
\Phi:\rm{Sym}^r(\tilde{{\C}}) \rightarrow \J({\tilde \C})
$$
is a projective bundle. Hence $\rmW^r_{\P} \rightarrow \P$ is also a projective bundle.

\begin{theorem}
(i) The induced pullback map
$$
\Br'(\P) \rightarrow \Br'(\rmW^r_{\P})
$$
is an isomorphism.
\vskip .2cm
(ii) If $s$ is the largest integer $\le r$, so that $\rmW^{s-1}_P \ne \rmW^r_P$, and $dim (\rmW^s_{\P}) \ge 4$, then $\Br(\rmW^r_P)_{\ell^n} = \Br(\rmW^s_P)_{\ell^n} \ra \Br(\rmW^{s-1}_P)_{\ell^n}$ is surjective
as long as $\ell$ is a prime different from the characteristic $p$ of the base field. Moreover this map will be an isomorphism
if $\rmW^{s-1}_P$ is smooth and the induced map $\Pic(\rmW^s_P)/\ell^n \ra \Pic(\rmW^{s-1}_P)/\ell^n $ is also an isomorphism; the last condition is
satisfied if the base field is  the complex numbers, $dim(\rmW^s_{\P})\geq 4$ and ${\rm W}^{s-1}_{\P}$ is smooth.
\end{theorem}
\begin{proof} 
(i) We invoke Gabber's theorem \cite[p. 193]{Ga},  to obtain the exact sequence:
$$
0\rightarrow [{\C}l(\alpha)] \rightarrow \Br'(\P) \rightarrow \Br'(\rmW^r_{\P}) \rightarrow 0.
$$
Here $Cl(\alpha)$ denotes the class of the projective bundle in the Brauer group of $\P$ and $[{\C}l(\alpha)]$ denotes the
 subgroup generated by this class. However, since the projective bundle is a projectivization of  a vector bundle, see \S \ref{Poincare}, it is a trivial class in the Brauer group of $\P$. Hence the pullback map is an isomorphism of Brauer groups, thereby proving (i). 
\vskip .2cm
(ii) Next observe that the complement $\rmW^s_P-\rmW^{s-1}_P$ is open in $\rmW^s_{\P}=\rmW^r_{\P}$, and therefore smooth. Observe also that $\rmW^s_{\P}-\rmW^{s-1}_{\P} = (\Sym^s(\tilde \C) - \Sym^{s-1}(\tilde \C))\cap \Phi^{-1}(\P)$ is affine, which holds
 since $\Phi^{-1}(\P)$ is closed in $\Sym^r(\tilde \C)$ and $\Sym^s(\tilde \C) - \Sym^{s-1}(\tilde \C)$ is affine.
Now the proof of the weak-Lefschetz isomorphism in Proposition ~\ref{weak.l} shows that the conclusions there hold even if the closed
subvariety $\rmY$ is not smooth, as long as the complement $\rmX-\rmY$ is smooth and affine. 
Therefore, if $\rmW^s_P$ is chosen as in (ii), the restriction map $\H^2_{et}(\rmW^s_{\P}, \mu_{\ell^n}) \ra \H^2_{et}(\rmW^{s-1}_{\P}, \mu_{\ell^n})$
is an isomorphism. In view of the short-exact sequence in ~\eqref{Kummer.seq.2}, this readily implies the surjection of the Brauer groups in (ii). The assertion that this map will become an isomorphism when the
restriction $\Pic(\rmW^s_{\P})/\ell^n \ra \Pic(\rmW^{s-1}_{\P})/\ell^n $ is also an isomorphism follows once again from the short-exact sequence ~\eqref{Kummer.seq.2}.
\vskip .2cm
Finally, the assertion that one has the required weak-Lefschetz isomorphism for the Picard groups when $k =\Cl$ is the Grothendieck-Lefschetz theorem: see \cite{SGA2} or \cite[Chapter IV, Corollary 3.3]{Hart}.
\end{proof}
\begin{remark} Clearly the above theorem proves  Theorem ~\ref{main.thm.2}. If one also knows the schemes ${\rm W}^s_{\P}$ are all smooth, one can
obtain stronger results, for example, that the conclusions in (ii) hold for all $s$ as long as $\dim({\rm W}^s_{\P}) \ge 4$.
 \end{remark}


\end{document}